\documentclass[11pt]{amsart}
\usepackage[
    left=1.15in,
    right=1.15in,
    top=1.0in,
    bottom=1.05in
]{geometry}
\usepackage{amsmath,amsthm,amssymb,mathtools,amsfonts,amsopn,amscd}
\usepackage{bm}
\usepackage{tikz}
\usepackage{tikz-cd} 
\usepackage{rotating}

\usepackage{nicematrix}
\usetikzlibrary{calc}
\usetikzlibrary{decorations.pathreplacing}

\AtBeginDocument{
    \tikzset{
        nicematrix/brace/.append style={
            line width=0.85pt
        }
    }
}

\usepackage{graphicx}
\usepackage{etoolbox}
\usepackage{enumitem}
\usepackage{longtable}
\usepackage{booktabs}
\usepackage{tabularx}
\usepackage{comment}
\usepackage{hyperref}
\usepackage{subfiles}
\usepackage{calc}
\usepackage[numbered]{bookmark}
\usepackage{mathrsfs}
\usepackage[sc]{mathpazo}
\usepackage{euscript}

\usepackage[noabbrev]{cleveref}
\crefname{section}{Section}{Sections}
\Crefname{section}{Section}{Sections}
\crefname{subsection}{Section}{Sections}
\Crefname{subsection}{Section}{Sections}
\crefname{equation}{Equation}{Equations}
\Crefname{equation}{Equation}{Equations}
\crefname{figure}{Figure}{Figures}
\Crefname{figure}{Figure}{Figures}
\crefname{table}{Table}{Tables}
\Crefname{table}{Table}{Tables}
\crefname{thm}{Theorem}{Theorems}
\Crefname{thm}{Theorem}{Theorems}
\crefname{lem}{Lemma}{Lemmas}
\Crefname{lem}{Lemma}{Lemmas}
\crefname{prop}{Proposition}{Propositions}
\Crefname{prop}{Proposition}{Propositions}
\crefname{cor}{Corollary}{Corollaries}
\Crefname{cor}{Corollary}{Corollaries}
\crefname{df}{Definition}{Definitions}
\Crefname{df}{Definition}{Definitions}
\crefname{ex}{Example}{Examples}
\Crefname{ex}{Example}{Examples}
\crefname{rmk}{Remark}{Remarks}
\Crefname{rmk}{Remark}{Remarks}
\crefname{conj}{Conjecture}{Conjectures}
\Crefname{conj}{Conjecture}{Conjectures}
\crefname{clm}{Claim}{Claims}
\Crefname{clm}{Claim}{Claims}

\AddToHook{env/lem/begin}{\crefalias{thm}{lem}}
\AddToHook{env/prop/begin}{\crefalias{thm}{prop}}
\AddToHook{env/cor/begin}{\crefalias{thm}{cor}}
\AddToHook{env/df/begin}{\crefalias{thm}{df}}
\AddToHook{env/ex/begin}{\crefalias{thm}{ex}}
\AddToHook{env/rmk/begin}{\crefalias{thm}{rmk}}
\AddToHook{env/clm/begin}{\crefalias{thm}{clm}}
\AddToHook{env/conj/begin}{\crefalias{thm}{conj}}

\setlist[enumerate,1]{label=(\roman*)}

\newif\ifusebiblatex

\definecolor{secondaryColor}{RGB}{0, 0, 170}
\hypersetup{
    colorlinks=true,
    linkcolor=blue,
    urlcolor=secondaryColor,
    citecolor=secondaryColor,
    linktoc=page,
}

\theoremstyle{plain}
\newtheorem{thm}{Theorem}[section]
\newtheorem{lem}[thm]{Lemma}
\newtheorem{prop}[thm]{Proposition}
\newtheorem{cor}[thm]{Corollary}
\newtheorem{conj}[thm]{Conjecture}

\theoremstyle{definition}
\newtheorem{df}[thm]{Definition}

\theoremstyle{remark}
\newtheorem{rmk}[thm]{Remark}

\newcommand{\ZZ}{\mathbb{Z}}
\newcommand{\DD}{\mathbb{D}}
\newcommand{\NN}{\mathbb{N}}
\newcommand{\QQ}{\mathbb{Q}}

\newcommand{\CC}{\mathbb{C}}
\newcommand{\HH}{\mathbb{H}}
\newcommand{\FF}{\mathbb{F}}

\newcommand{\Bcal}{\mathcal{B}}

\newcommand{\Dcal}{\mathcal{D}}
\newcommand{\Ecal}{\mathcal{E}}

\newcommand{\Ical}{\mathcal{I}}

\newcommand{\Lcal}{\mathcal{L}}

\newcommand{\Zcal}{\mathcal{Z}}

\newcommand{\Frac}{\operatorname{Frac}}

\newcommand{\Mat}{\operatorname{Mat}}

\newcommand{\tr}{\operatorname{tr}}
\newcommand{\Tr}{\operatorname{Tr}}

\newcommand{\Nr}{\operatorname{Nr}}

\newcommand{\ord}{\operatorname{ord}}

\let\oldforall\forall
\renewcommand{\forall}{\oldforall \: }
\let\oldexist\exists
\renewcommand{\exists}{\oldexist \: }

\newcommand{\dep}{\operatorname{dep}}
\newcommand{\wt}{\operatorname{wt}}

\newcommand{\bfz}{\mathbf{z}}

\newcommand{\ov}[1]{\overline{#1}}

\newcommand{\Li}{\operatorname{Li}}
\newcommand{\bs}{\mathfrak{s}}
\newcommand{\bu}{\mathbf{u}}

\allowdisplaybreaks

\newcommand{\bT}{\mathbf T}

\newcommand{\Conv}{{\mathrm{Conv}}}

\makeatletter

\newcommand{\myToC}{{
    \renewcommand{\contentsname}{}
    \@starttoc{toc}{\contentsname}
}}

\patchcmd{\@tocline}
    {\hfil}
    {\leaders\hbox{\,.\,}\hfil}
    {}{}
\makeatother

\title[Graded Structures of
$v$-adic Carlitz Multiple Polylogarithms and Multiple Zeta Values]
{Graded Structures of $v$-adic Carlitz Multiple Polylogarithms and $v$-adic Multiple Zeta Values}

\author{Hung-Chun Tsui}
\address{(Hung-Chun Tsui) Department of Mathematics, National Tsing Hua University, No. 101, Sec. 2, Guangfu Rd., East Dist., Hsinchu City 300044, Taiwan (R.O.C.)}
\email{hctsui@gapp.nthu.edu.tw}

\date{\today}

\subjclass[2020]{Primary 11M38; Secondary 11M32, 11R58}
\keywords{Function field, $v$-adic multiple zeta values, Carlitz multiple polylogarithms}

\thanks{}

\begin{document}

\begin{abstract}
In this paper, we study $v$-adic Carlitz multiple polylogarithms and $v$-adic multiple
zeta values. We prove that, for every finite place $v$ of $\FF_q(\theta)$, the natural
weight gives a graded structure on the algebras generated by these values. Moreover, we
show that every $\ov{\FF_q(\theta)}$-linear relation among these values arises from an
$\FF_q(\theta)$-linear relation. In particular, our results establish a function field
analogue of the Furusho--Yamashita direct sum conjecture and extend Chen's work from
finite places of degree one to arbitrary finite places.
\end{abstract}

\maketitle
\tableofcontents

\section{Introduction}
\subsection{Classical background}\label{sec:classical-background}
Let $\NN$ denote the set of positive integers and $\ZZ$ the set of integers. We let
\[
    \Ical:=\{\varnothing\}\cup\bigcup_{r\geq1}\NN^r
\]
denote the set of indices and define
\[
    \dep(\bs):=r,
    \qquad \wt(\bs):=s_1+\cdots+s_r,
    \qquad \dep(\varnothing)=\wt(\varnothing)=0
\]
for every $\bs=(s_1,\ldots,s_r)\in\Ical$. For an admissible non-empty index
$\bs=(s_1,\ldots,s_r)\in\Ical$, namely $s_1\geq2$, the classical \emph{multiple zeta
value} is defined by
\[
    \zeta(\bs) := \sum_{n_1>\cdots>n_r\geq1} \frac{1}{n_1^{s_1}\cdots n_r^{s_r}}.
\]
In depth one, these values recover the special values of the Riemann zeta function at
positive integers. We refer the reader to \cite{Zha16} for further background.

For $w\geq1$, let $\Zcal_w$ denote the $\QQ$-vector space spanned by all multiple zeta
values of weight $w$, and let $\Zcal$ be the $\QQ$-algebra generated by all multiple
zeta values. Multiple zeta values satisfy both the stuffle and shuffle product relations
\cite[\S3]{Zha16}. In particular, we have
$\Zcal_{w_1}\Zcal_{w_2}\subseteq\Zcal_{w_1+w_2}$ for any $w_1,w_2\geq1$.

Multiple zeta values can also be viewed as special values of one-variable multiple
polylogarithms. Furusho \cite{Fur04} constructed $p$-adic multiple polylogarithms by
means of Coleman's $p$-adic iterated integration \cite{Col82} and defined the
corresponding \emph{$p$-adic multiple zeta values} $\zeta_p(\bs)$ through suitable
limits at $z=1$. Let $\Zcal_p$ be the $\QQ$-algebra generated by all $p$-adic multiple
zeta values, and for $w\geq1$, let $\Zcal_{p,w}$ denote the $\QQ$-vector space spanned
by the $p$-adic multiple zeta values of weight $w$. The $p$-adic multiple zeta values
also satisfy the stuffle and shuffle product relations \cite{BF06}. In particular, we
have $\Zcal_{p,w_1}\Zcal_{p,w_2} \subseteq \Zcal_{p,w_1+w_2}$ for any $w_1,w_2\geq1$.

Furusho and Yamashita formulated the following direct sum conjecture for $p$-adic
multiple zeta values \cite[Conjecture~5]{Yam10}.

\begin{conj}[Furusho--Yamashita]\label{conj:FY}
For every prime number $p$,
\[
    \Zcal_p = \QQ\oplus \bigoplus_{w\geq1}\Zcal_{p,w}
\]
forms a graded $\QQ$-algebra.
\end{conj}

This conjecture is the $p$-adic analogue of the celebrated Goncharov's direct sum
conjecture for classical multiple zeta values, which predicts that the weight gives a
grading on the algebra of multiple zeta values (see, for example,
\cite[Conjecture~4.2]{Gon97}).

\subsection{The function field counterpart}\label{sec:v-adicMZVs}

Let $p$ be a prime number, let $q$ be a power of $p$, and put
\[
    A=\FF_q[\theta],
    \qquad K=\FF_q(\theta).
\]
Let $K_\infty=\FF_q(\!(1/\theta)\!)$ be the completion of $K$ at the infinite place and
let $\CC_\infty$ be the completion of an algebraic closure of $K_\infty$. We denote by
$|\cdot|_\infty$ the $\infty$-adic absolute value on $\CC_\infty$, normalized such that
$|\theta|_\infty=q$. Fix an algebraic closure $\ov{K}$ of $K$ together with an embedding
$\ov{K}\hookrightarrow\CC_\infty$.

For $\bs=(s_1,\ldots,s_r)\in\Ical$, the \emph{Carlitz multiple polylogarithm} at $\bs$
is defined in \cite{Chang14} by
\[
    \Li_{\bs}(\bfz) := \sum_{i_1>\cdots>i_r\geq0}
    \frac{z_1^{q^{i_1}}\cdots z_r^{q^{i_r}}}{L_{i_1}^{s_1}\cdots L_{i_r}^{s_r}} \in
    K[\![z_1,\ldots,z_r]\!],
\]
where $L_0:=1$ and $L_i:=\prod_{j=1}^i(\theta-\theta^{q^j})$ for all $i\geq1$. This
series converges on the domain
\[
    \DD_{\bs,\infty}^{\Conv} := \left\{\bu=(u_1,\ldots,u_r)\in\CC_\infty^r \ \middle|\
    |u_i|_\infty<q^{\frac{s_iq}{q-1}} \text{ for }1\leq i\leq r\right\}.
\]
For $w\geq1$, let $\Lcal_{\infty,w}$ (resp. $\overline{\Lcal}_{\infty,w}$) denote the
$K$-vector space (resp. $\ov{K}$-vector space) spanned by all Carlitz multiple
polylogarithms of weight $w$ at algebraic points in their domains of convergence, and
let $\Lcal_\infty$ (resp. $\overline{\Lcal}_\infty$) be the $K$-vector space (resp.
$\ov{K}$-vector space) spanned by $1$ and all such values. The harmonic product
relations (see \cite[\S5.2]{Chang14}) imply that $\Lcal_\infty$ and
$\overline{\Lcal}_\infty$ are respectively a $K$-algebra and an $\ov{K}$-algebra, with
multiplication preserving the total weight.

On the other hand, Thakur \cite{Thakur04} introduced the \emph{$\infty$-adic multiple
zeta values}
\[
    \zeta_A(\bs) := \sum_{\substack{a_1,\ldots,a_r\in A\ {\rm monic}\\
    \deg a_1>\cdots>\deg a_r\geq0}} \frac{1}{a_1^{s_1}\cdots a_r^{s_r}} \in K_\infty
\]
for every non-empty index $\bs=(s_1,\ldots,s_r)\in\Ical$. For $w\geq1$, let
$\Zcal_{\infty,w}$ (resp. $\overline{\Zcal}_{\infty,w}$) denote the $K$-vector space
(resp. $\ov{K}$-vector space) spanned by all $\infty$-adic multiple zeta values of
weight $w$, and let $\Zcal_\infty$ (resp. $\overline{\Zcal}_\infty$) be the $K$-vector
space (resp. $\ov{K}$-vector space) spanned by $1$ and all such values. The $q$-shuffle
product relations established by Thakur \cite{Thakur10} imply that $\Zcal_\infty$ and
$\overline{\Zcal}_\infty$ are respectively a $K$-algebra and an $\ov{K}$-algebra, with
multiplication preserving the total weight. Chang \cite{Chang14} showed that every
$\infty$-adic multiple zeta value can be expressed as a $K$-linear combination of
Carlitz multiple polylogarithms of the same weight at algebraic points. He further
proved that
\[
    \overline{\Lcal}_\infty = \ov{K}\oplus\bigoplus_{w\geq1}\overline{\Lcal}_{\infty,w}
\]
forms a graded $\ov{K}$-algebra. Moreover, Chang showed that the canonical map
\[
    \ov{K}\otimes_K\Lcal_\infty \xrightarrow{\ \sim\ } \overline{\Lcal}_\infty
\]
is an isomorphism. Consequently, the corresponding results also hold for $\Zcal_\infty$
and $\overline{\Zcal}_\infty$. The aim of this paper is to establish the $v$-adic
analogues of these results for arbitrary finite places $v$ of $K$.

To formulate these analogues, let $v$ be a fixed finite place of $K$, let $K_v$ be the
completion of $K$ at $v$, and let $\CC_v$ be the completion of an algebraic closure of
$K_v$. We denote by $|\cdot|_v$ the normalized $v$-adic absolute value on $\CC_v$ and
fix an embedding $\ov{K}\hookrightarrow\CC_v$. For a non-empty index
$\bs=(s_1,\ldots,s_r)$, put
\[
    \DD_{\bs,v}^{\Conv} := \left\{\bu=(u_1,\ldots,u_r)\in\CC_v^r \ \middle|\ |u_1|_v<1,\quad
    |u_i|_v\leq1\ \text{ for }2\leq i\leq r\right\}.
\]
Then for $\bu\in(\ov{K}^{\times})^{\dep(\bs)} \cap\DD_{\bs,v}^{\Conv}$, the series
defining $\Li_{\bs}(\bu)$ also converges $v$-adically, and we denote its value by
$\Li_{\bs}(\bu)_v$. Chang and Mishiba \cite{CM19} extended $v$-adic Carlitz multiple
polylogarithms and Carlitz multiple star polylogarithms beyond their domains of
convergence by means of logarithms of suitable $t$-modules. Using an $\infty$-adic
presentation of multiple zeta values in terms of Carlitz multiple star polylogarithms at
integral points, they then defined the corresponding \emph{$v$-adic multiple zeta
values} $\zeta_A(\bs)_v$ for every non-empty index $\bs\in\Ical$ (see
\cite[\S6.1]{CM21}). Chang, Chen, and Mishiba \cite[Theorem~1.2.3]{CCM22} further proved
that these values satisfy the same $q$-shuffle product relations as their $\infty$-adic
counterparts.

For $w\geq1$, let $\Lcal_{v,w}$ (resp. $\overline{\Lcal}_{v,w}$) denote the $K$-vector
space (resp. $\ov{K}$-vector space) spanned by all values
\[
    \Li_{\bs}(\bu)_v,
    \qquad \wt(\bs)=w,
    \qquad \bu\in(\ov{K}^{\times})^{\dep(\bs)} \cap\DD_{\bs,v}^{\Conv},
\]
and let $\Lcal_v$ (resp. $\overline{\Lcal}_v$) be the $K$-vector space (resp.
$\ov{K}$-vector space) spanned by $1$ and all such values. Likewise, let $\Zcal_{v,w}$
(resp. $\overline{\Zcal}_{v,w}$) denote the $K$-vector space (resp. $\ov{K}$-vector
space) spanned by all $v$-adic multiple zeta values of weight $w$, and let $\Zcal_v$
(resp. $\overline{\Zcal}_v$) be the $K$-vector space (resp. $\ov{K}$-vector space)
spanned by $1$ and all such values. The harmonic product relations and $q$-shuffle
product relations imply that $\Lcal_v$ and $\Zcal_v$ are $K$-algebras, while
$\overline{\Lcal}_v$ and $\overline{\Zcal}_v$ are $\ov{K}$-algebras, with multiplication
preserving the total weight.

When $v$ is a finite place of degree one, Chen \cite{ChenABP} established the $v$-adic
analogue of Chang's results for Carlitz multiple polylogarithms at algebraic points and,
as an application, the corresponding results for $v$-adic multiple zeta values. The
purpose of the present paper is to extend these results to finite places of arbitrary
degree.

\subsection{Main results}

Our main result establishes the function field analogue of the Furusho--Yamashita
conjecture and, in particular, generalizes Chen's result \cite{ChenABP} from degree one
finite places to arbitrary finite places.

\begin{thm}\label{thm:main}
For every finite place $v$ of $K$,
\[
    \overline{\Lcal}_v = \ov{K}\oplus\bigoplus_{w\geq1}\overline{\Lcal}_{v,w}
\]
forms a graded $\ov{K}$-algebra. Moreover, the canonical map
\[
    \ov{K}\otimes_K\Lcal_v \xrightarrow{\ \sim\ } \overline{\Lcal}_v
\]
is an isomorphism.
\end{thm}

The corresponding result for $v$-adic multiple zeta values is as follows.

\begin{thm}\label{thm:mainMZV}
For every finite place $v$ of $K$,
\[
    \overline{\Zcal}_v = \ov{K}\oplus\bigoplus_{w\geq1}\overline{\Zcal}_{v,w}
\]
forms a graded $\ov{K}$-algebra. Moreover, the canonical map
\[
    \ov{K}\otimes_K\Zcal_v \xrightarrow{\ \sim\ } \overline{\Zcal}_v
\]
is an isomorphism.
\end{thm}

As an immediate consequence, we obtain the following transcendence statement.

\begin{cor}\label{cor:transcendence}
Let $v$ be a finite place of $K$. Then the following assertions hold.
\begin{enumerate}
\item For every non-empty index $\bs\in\Ical$ and every
$\bu\in(\ov{K}^{\times})^{\dep(\bs)}\cap\DD_{\bs,v}^{\Conv}$, the value
$\Li_{\bs}(\bu)_v$ is either zero or transcendental over $K$.

\item For every non-empty index $\bs\in\Ical$, the $v$-adic multiple
zeta value $\zeta_A(\bs)_v$ is either zero or transcendental over $K$.
\end{enumerate}
\end{cor}

The proofs of our main results build on the strategy developed by Chang \cite{Chang14}
and its $v$-adic adaptation by Chen \cite{ChenABP}. However, extending this approach to
finite places of higher degree requires new ingredients. We construct a two-variable
deformation of $v$-adic Carlitz multiple polylogarithms using an auxiliary function
field $H_v$ of degree $d$ over $\ov{K}(t)$ (see \S\ref{subsec:auxiliary-polynomials} and
\S\ref{subsec:L-series}). The key idea is to use trace coordinates (see
\eqref{eq:trace-coordinates}) to convert the Frobenius difference equations satisfied by
these deformations into Frobenius systems over $\ov{K}[t]$ to which the $v$-adic
Anderson--Brownawell--Papanikolas criterion (see \cref{thm:vadic-ABP}) can be applied
(see \cref{prop:trace-entire,prop:Frobenius-system}). When $d=1$, our construction
reduces to Chen's. The proofs of the main theorems are given in
Appendix~\ref{sec:proof-main}.

Beyond the graded algebra structures established above, we expect that the method of
trace coordinates developed in this paper may also be useful in studying algebraic
independence problems for certain special functions and their values in positive
characteristic, particularly $v$-adic Carlitz multiple polylogarithms. We hope to
explore this direction in future work.

\begin{rmk}
After the present paper was posted on arXiv, the author became aware of a preprint
written in French by Adam and Estienne \cite{AE26}, which had been posted on viXra about
one month earlier and independently establishes the same main results. The two works
were carried out independently, and their proofs use totally different methods. The
approach of Adam and Estienne is based on Mahler's method, whereas ours relies on
two-variable deformations, trace coordinates, and the $v$-adic
Anderson--Brownawell--Papanikolas criterion.
\end{rmk}

\section*{Acknowledgments}
The author would like to thank Chieh-Yu Chang for helpful discussions and valuable
suggestions. The author also gratefully acknowledges the support of the National Science
and Technology Council over the past few years under grant no. 113-2628-M-007-004.

\section*{Statement on AI usage}
The author only used AI tools (ChatGPT 6 Sol) to help improve the language and
presentation of the manuscript and to check for grammatical and mathematical errors. AI
tools were not used to develop any original mathematical arguments.
\section{Deformations of \texorpdfstring{$v$}{v}-adic Carlitz Multiple Polylogarithms}
\label{sec:vadic-setting}
In this section, we prepare the tools needed for the construction of Frobenius systems
associated with $v$-adic Carlitz multiple polylogarithms. We first recall the $v$-adic
Anderson--Brownawell--Papanikolas criterion in \S\ref{subsec:notations}. In
\S\ref{subsec:auxiliary-polynomials}, we introduce an auxiliary function field and
certain polynomials used to construct a two-variable deformation of $v$-adic Carlitz
multiple polylogarithms. Finally, in \S\ref{subsec:L-series}, we establish the Frobenius
difference equation satisfied by this deformation.

\subsection{\texorpdfstring{$v$}{v}-adic Anderson--Thakur series and ABP criterion}\label{subsec:notations}
Throughout this paper, we also denote by $v$ the monic irreducible polynomial in $A$
associated with the finite place $v$ when there is no confusion. For every $n\in\ZZ$, we
define the $n$-fold Frobenius twist
\[
    (-)^{(n)}: \CC_v(\!(t)\!) \longrightarrow \CC_v(\!(t)\!)
\]
by
\[
    \left(\sum_{i} a_{i}t^i\right)^{(n)} := \sum_{i} a_{i}^{q^n}t^i.
\]
Then the map $(-)^{(n)}$ is an $\FF_q(t)$-algebra automorphism of $\CC_v(\!(t)\!)$.
Following \cite[\S2.1]{ChenABP}, the ring of $v$-adic entire series is defined by
\[
    \Ecal_v := \left\{\sum_{i=0}^{\infty}a_it^i\in\ov{K}[\![t]\!] \ \middle|\
    \lim_{i\to\infty}|a_i|_v^{1/i}=0,\ [K_v(a_0,a_1,\ldots):K_v]<\infty\right\}.
\]
Thus every element of $\Ecal_v$ defines an entire function on $\CC_v$, and its
coefficients are contained in a finite extension of $K_v$.

We also recall the $v$-adic Anderson--Thakur series (see
\cite[Definition~3.1.1]{ChenABP}; see also \cite[Example~5.7]{HK20})
\[
    \Omega_v(t) := \prod_{i\geq1} \left(1-v^{q^i}t\right).
\]
By \cite[Proposition~3.1.3]{ChenABP}, one has
\begin{equation}\label{eq:v-adic AT}
    \Omega_{v}(t)\in\Ecal_v,
    \qquad \Omega_{v}^{(-1)}(t)=(1-v t)\Omega_{v}(t),
    \qquad \Omega_{v}^{(n)}(t)=\frac{\Omega_{v}(t)}{\prod_{j=1}^{n}(1-v^{q^j}t)} \quad(n\geq0).
\end{equation}

The Frobenius difference equation satisfied by $\Omega_v(t)$ will play an important role
in our construction of Frobenius systems. To study the linear independence of Carlitz
multiple polylogarithms, we will use the following $v$-adic analogue of the
Anderson--Brownawell--Papanikolas criterion (see \cite{ABP04}) and its refinement (see
\cite{Chang09}), established by Chen \cite{ChenABP}.

\begin{thm}[{\cite[Theorem~2.2.2]{ChenABP}}]
\label{thm:vadic-ABP}
Let
\[
    \Phi(t)\in\Mat_{\ell}(\ov{K}[t]),
    \qquad \psi(t)\in\Mat_{\ell\times1}(\Ecal_v)
\]
satisfy $\psi^{(-1)}=\Phi\psi$. Let
$\gamma\in \ov{K}^{\times}\setminus\ov{\FF}_q^{\times}$ and suppose that
\[
    \det\Phi\left(\gamma^{(-i)}\right)\neq0
    \qquad \text{for every }i\in\NN.
\]
If $\rho\in\Mat_{1\times\ell}(\ov{K})$ satisfies
\[
    \rho\psi(\gamma)=0,
\]
then there exists $P(t)\in\Mat_{1\times\ell}(\ov{K}[t])$ such that
\[
    P(\gamma)=\rho
    \qquad\text{and}\qquad P(t)\psi(t)=0.
\]
\end{thm}

\subsection{Auxiliary function field and polynomials}\label{subsec:auxiliary-polynomials}
Throughout the paper, we put $d:=\deg v$ and write
$v=v(\theta)=\theta^d+a_{d-1}\theta^{d-1}+\cdots+a_0\in A$ where $a_i\in\FF_q$. We
consider the function field
\[
    H_v := \Frac\left(\ov{K}[t,T]/(tv(T)-1)\right).
\]
Thus the variables $T$ and $t$ satisfy
\begin{equation}\label{eq:Hv-relation}
    t^{-1}=v(T)=T^d+a_{d-1}T^{d-1}+\cdots+a_0.
\end{equation}
In particular, we identify $H_v$ with $\ov{K}(T)$ via $t=v(T)^{-1}$. Since $v(X)-t^{-1}$
is the minimal polynomial of $T$ over $\ov{K}(t)$ and $v'(X)\neq0$, the extension
$H_v/\ov{K}(t)$ is separable of degree $d$. Consequently, $\{1,T,\ldots,T^{d-1}\}$ is an
$\ov{K}(t)$-basis of $H_v$. To construct the two-variable deformation of $v$-adic
Carlitz multiple polylogarithms, we consider the scalar extension
\[
    \HH_v := H_v\otimes_{\ov{K}(t)}\CC_v(\!(t)\!) \cong \CC_v(\!(t)\!)[T]/(v(T)-t^{-1}).
\]
The basis $\{1,T,\ldots,T^{d-1}\}$ of $H_v$ over $\ov{K}(t)$ remains a basis of $\HH_v$
over $\CC_v(\!(t)\!)$. Moreover, $\HH_v$ is endowed with the natural $t$-adic topology
induced by $\CC_v(\!(t)\!)$.

For every $n\in\ZZ$, we extend the $n$-fold Frobenius twist to $H_v$ and $\HH_v$ by
fixing $T$. Since $v(T)\in\FF_q[T]$, the relation \eqref{eq:Hv-relation} is preserved
under Frobenius twisting, so these extensions are well-defined automorphisms.
Explicitly, for $F=\sum_{j=0}^{d-1}f_j(t)T^j\in\HH_v$, we have
\[
    F^{(n)} = \sum_{j=0}^{d-1}f_j^{(n)}(t)T^j\in \HH_v.
\]

We now introduce the auxiliary polynomials used to construct the two-variable
deformation of $v$-adic Carlitz multiple polylogarithms in \S\ref{subsec:L-series}.
These polynomials are chosen so that the deformation satisfies a Frobenius difference
equation and specializes to $v$-adic Carlitz multiple polylogarithms at
$(T,t)=(\theta,v^{-1})$. For $a\in\ov{K}$, define the auxiliary polynomial
\[
    \Delta_v(T,a) := \frac{v(T)-v(a)}{T-a} \in\ov{K}[T].
\]
Since $v(T)\in\FF_q[T]$, it satisfies
\begin{equation}\label{eq:Delta-twist}
    \Delta_v(T,a)^{(n)} = \Delta_v(T,a^{q^n})
\end{equation}
for every $n\in\ZZ$. In particular, put $\delta_v(T):=\Delta_v(T,\theta)$. For every
$i\geq0$, define
\[
    \Dcal_{v,0}(T):=1,
    \qquad \Dcal_{v,i}(T) := \prod_{j=1}^i \Delta_v(T,\theta^{q^j}) \quad (i\geq1).
\]
The following lemma gives the Frobenius twisting relation and the value of
$\Dcal_{v,i}(T)$ at $T=\theta$, which will be used in the subsequent constructions.

\begin{lem}\label{lem:Dvi}
For every $i\geq1$, we have
\begin{equation}\label{eq:Dvi-twist}
    \Dcal_{v,i}^{(-1)}(T) = \delta_v(T)\Dcal_{v,i-1}(T),
\end{equation}
and
\begin{equation}\label{eq:Dvi-theta}
    \Dcal_{v,i}(\theta) = \prod_{j=1}^i \frac{v-v^{q^j}}{\theta-\theta^{q^j}}.
\end{equation}
\end{lem}

\begin{proof}
By \eqref{eq:Delta-twist}, we have
\[
    \begin{aligned}
        \Dcal_{v,i}^{(-1)}(T) = \prod_{j=1}^i \Delta_v(T,\theta^{q^{j-1}})= \Delta_v(T,\theta)
        \prod_{j=1}^{i-1} \Delta_v(T,\theta^{q^j})= \delta_v(T)\Dcal_{v,i-1}(T),
    \end{aligned}
\]
which proves \eqref{eq:Dvi-twist}. On the other hand, since $v\in\FF_q[\theta]$, we have
$v(\theta^{q^j})=v(\theta)^{q^j}=v^{q^j}$. Therefore, we obtain
\[
    \Delta_v(\theta,\theta^{q^j}) = \frac{v-v^{q^j}}{\theta-\theta^{q^j}},
\]
and \eqref{eq:Dvi-theta} follows immediately from the definition.
\end{proof}

\subsection{The series \texorpdfstring{$\mathfrak{L}_{\bs,\bu}^{[v]}(T,t)$}{Lv(T,t)}}
\label{subsec:L-series}
We now use the auxiliary polynomials introduced in \S\ref{subsec:auxiliary-polynomials}
to construct a two-variable deformation of $v$-adic Carlitz multiple polylogarithms.

\begin{df}
Let $\bs=(s_1,\ldots,s_r)\in\Ical$ be non-empty and let
$\bu=(u_1,\ldots,u_r)\in (\ov{K}^{\times})^r\cap\DD_{\bs,v}^{\Conv}$. For
$0\leq\ell\leq r$, put
\[
    w_\ell:=s_1+\cdots+s_\ell,
    \qquad w_0:=0,
    \qquad w:=w_r=\wt(\bs).
\]
We define
\begin{equation}\label{eq:L-series}
    \mathfrak{L}_{\bs,\bu}^{[v]}(T,t) := \sum_{i_1>\cdots>i_r\geq0} \prod_{\ell=1}^r
    \left((\Omega_v^{s_\ell}u_\ell)^{(i_\ell)} \Dcal_{v,i_\ell}(T)^{s_\ell}
    t^{i_\ell s_\ell}\right).
\end{equation}
For the empty index, we set $\mathfrak{L}_{\varnothing,\varnothing}^{[v]}(T,t):=1$ by
convention.
\end{df}
\begin{rmk}
The use of these deformation series goes back to Papanikolas \cite{Pap08}, who
constructed a deformation of the Carlitz logarithm. Chang and Yu \cite{CY07} extended
this construction to Carlitz polylogarithms. For higher depths, using the
Anderson--Thakur polynomials (see \cite{AT90}), Anderson and Thakur \cite{AT09}
constructed series of the type \eqref{eq:L-series}, whose specializations yield multiple
zeta values at $t=\theta$. Chang \cite{Chang14} developed the corresponding construction
for Carlitz multiple polylogarithms, and Chen \cite{ChenABP} adapted it to the $v$-adic
setting for finite places of degree one. Our construction extends these series to finite
places of arbitrary degree by introducing the auxiliary polynomials $\Dcal_{v,i}(T)$.
\end{rmk}

\begin{rmk}
If $\deg v=1$, then $\Dcal_{v,i}(T)=1$ for every $i\geq0$. Consequently,
$\mathfrak{L}_{\bs,\bu}^{[v]}(T,t)$ is independent of $T$ and coincides with Chen's
series in \cite[Definition~3.2.1]{ChenABP}.
\end{rmk}

We first verify that the series in \eqref{eq:L-series} indeed defines an element of
$\HH_v$.

\begin{lem}\label{lem:L-in-Kv}
For every non-empty $\bs=(s_1,\ldots,s_r)\in\Ical$ and
$\bu\in(\ov{K}^{\times})^r\cap\DD_{\bs,v}^{\Conv}$, the series
$\mathfrak{L}_{\bs,\bu}^{[v]}(T,t)$ converges in $\HH_v$ with respect to the $t$-adic
topology.
\end{lem}

\begin{proof}
Consider a summand of \eqref{eq:L-series} indexed by $i_1>\cdots>i_r\geq0$, and put
$n:=i_1s_1+\cdots+i_rs_r$. Since $\deg_T\Dcal_{v,i}=i(d-1)$, the polynomial
\[
    P(T):= \prod_{\ell=1}^r \Dcal_{v,i_\ell}(T)^{s_\ell}
\]
has degree at most $(d-1)n$. By \eqref{eq:Hv-relation}, we have
$T^d=t^{-1}-a_{d-1}T^{d-1}-\cdots-a_0$, where $a_j\in\FF_q$. Inductively reducing powers
of $T$ shows that $T^m$ can be expressed in the basis $\{1,T,\ldots,T^{d-1}\}$ with
coefficients of $t$-adic order at least $-\lfloor m/d\rfloor$. Consequently,
\[
    t^nP(T) \in t^{\lceil n/d\rceil} \sum_{j=0}^{d-1}\CC_v[\![t]\!]T^j.
\]
Since $(\Omega_v^{s_\ell}u_\ell)^{(i_\ell)} \in\CC_v[\![t]\!]$ for every $\ell$, after
expanding the original summand with respect to the basis $\{1,T,\ldots,T^{d-1}\}$, its
coefficients also have $t$-adic order at least $\lceil n/d\rceil$. As $i_1\to\infty$, we
have $n\to\infty$, which proves the $t$-adic convergence of
$\mathfrak{L}_{\bs,\bu}^{[v]}(T,t)$ in $\HH_v$.
\end{proof}

We next establish the Frobenius difference equation satisfied by
$\mathfrak{L}_{\bs,\bu}^{[v]}(T,t)$.

\begin{prop}\label{prop:L-Frobenius}
Let $\bs=(s_1,\ldots,s_r)\in\Ical$ be non-empty and
$\bu=(u_1,\ldots,u_r)\in(\ov{K}^{\times})^r\cap\DD_{\bs,v}^{\Conv}$. Put
$\bs'=(s_1,\ldots,s_{r-1})$ and $\bu'=(u_1,\ldots,u_{r-1})$. Then
\begin{align}\label{eq:twist-L-series}
    \left(\mathfrak{L}_{\bs,\bu}^{[v]}(T,t)\right)^{(-1)} ={}& u_r^{q^{-1}}
    (1-vt)^{s_r}\Omega_v^{s_r} \left(t\delta_v(T)\right)^{w_{r-1}}
    \mathfrak{L}_{\bs',\bu'}^{[v]}(T,t) + \left(t\delta_v(T)\right)^w
    \mathfrak{L}_{\bs,\bu}^{[v]}(T,t).
\end{align}
\end{prop}

\begin{proof}
By \eqref{eq:L-series}, the left-hand side of \eqref{eq:twist-L-series} equals
\begin{align*}
    \sum_{i_1>\cdots>i_{r-1}>i_r=0 } \prod_{\ell=1}^r \left((\Omega_v^{s_\ell}u_\ell)^{(i_\ell)}
    \Dcal_{v,i_\ell}(T)^{s_\ell} t^{i_\ell s_\ell}\right)^{(-1)} +\sum_{i_1>\cdots>i_r\geq1 }
    \prod_{\ell=1}^r \left((\Omega_v^{s_\ell}u_\ell)^{(i_\ell)} \Dcal_{v,i_\ell}(T)^{s_\ell}
    t^{i_\ell s_\ell}\right)^{(-1)}.
\end{align*}

For the first summand, since $i_r=0$, we have $i_\ell\geq1$ for every $\ell<r$. Note
that by \eqref{eq:v-adic AT},
\[
    (\Omega_v^{s_r}u_r)^{(-1)} = u_r^{q^{-1}} (1-vt)^{s_r}\Omega_v^{s_r}
\]
Shifting $i_\ell$ to $i_\ell-1$ for $\ell<r$, by \cref{lem:Dvi}, the first summand
equals
\[
    u_r^{q^{-1}} (1-vt)^{s_r}\Omega_v^{s_r} \left(t\delta_v(T)\right)^{w_{r-1}}
    \mathfrak{L}_{\bs',\bu'}^{[v]}(T,t).
\]

For the second summand, since $i_r\geq1$, we have $i_\ell\geq1$ for every $\ell$. After
replacing each $i_\ell$ by $i_\ell-1$, by \cref{lem:Dvi} again, it equals
\[
    \left(t\delta_v(T)\right)^{w} \mathfrak{L}_{\bs,\bu}^{[v]}(T,t)
\]
Combining the two cases proves the result.
\end{proof}

We now write (cf. \cite[Corollary 5.1.4]{ABP04})
\[
    \widetilde{\pi}_v:=\Omega_v(v^{-1})\in\CC_v^\times.
\]
Then $\mathfrak{L}_{\bs,\bu}^{[v]}(T,t)$ specializes to the corresponding $v$-adic
Carlitz multiple polylogarithm at $(T,t)=(\theta,v^{-1})$ as follows.

\begin{prop}\label{prop:L-specialization}
Let $\bs=(s_1,\ldots,s_r)\in\Ical$ be non-empty and
$\bu=(u_1,\ldots,u_r)\in(\ov{K}^{\times})^r\cap\DD_{\bs,v}^{\Conv}$. We have
\begin{equation}\label{eq:L-specialization}
    \mathfrak{L}_{\bs,\bu}^{[v]}(\theta,v^{-1}) = \widetilde{\pi}_v^{\,w} \Li_{\bs}(\bu)_v.
\end{equation}
\end{prop}

\begin{proof}
By \eqref{eq:v-adic AT}, for every $i\geq0$,
\[
    \Omega_v^{(i)}(t) = \frac{\Omega_v(t)}{\prod_{j=1}^i(1-v^{q^j}t)},
\]
and hence
\[
    \Omega_v^{(i)}(v^{-1})v^{-i} = \frac{\widetilde{\pi}_v}{\prod_{j=1}^i(v-v^{q^j})}.
\]
By \cref{lem:Dvi}, we have
\[
    \Dcal_{v,i}(\theta) = \prod_{j=1}^i \frac{v-v^{q^j}}{\theta-\theta^{q^j}}.
\]
Therefore, we obtain
\[
    \Omega_v^{(i)}(v^{-1})v^{-i} \Dcal_{v,i}(\theta) = \frac{\widetilde{\pi}_v}{L_i}.
\]
Applying this identity to each factor in \eqref{eq:L-series} gives
\eqref{eq:L-specialization}.
\end{proof}

\section{Trace Coordinates and Frobenius Systems}
To apply the $v$-adic Anderson--Brownawell--Papanikolas criterion (see
\cref{thm:vadic-ABP}), we use trace coordinates to convert the two-variable deformation
into $v$-adic entire series satisfying Frobenius difference equations over $\ov{K}[t]$.
In \S\ref{subsec:trace-coordinates}, we introduce trace coordinates and establish the
required entireness. We then construct the corresponding Frobenius systems in
\S\ref{subsec:Frobenius-equations} and study their specializations in
\S\ref{subsec:frobenius-fibers}.
\subsection{Trace coordinates}\label{subsec:trace-coordinates}
Recall that $H_v/\ov{K}(t)$ is a separable extension of degree $d$ with basis
$\{1,T,\ldots,T^{d-1}\}$. Let
\[
    \Tr_{H_v/\ov{K}(t)}: H_v\longrightarrow\ov{K}(t)
\]
denote the field trace, which uniquely extends to a $\CC_v(\!(t)\!)$-linear map
\[
    \Tr:\HH_v\longrightarrow\CC_v(\!(t)\!).
\]
We define the trace coordinate map
$\bT:\HH_v\longrightarrow \Mat_{d\times1}\left(\CC_v(\!(t)\!)\right)$ by
\begin{equation}\label{eq:trace-coordinates}
    \bT(F) :=
    \begin{pmatrix}
        \Tr(F)\\
        \Tr(TF)\\
        \vdots\\
        \Tr(T^{d-1}F)
    \end{pmatrix}
    \qquad (F\in\HH_v).
\end{equation}
We call $\bT(F)$ the trace coordinates of $F$. Since the trace pairing is nondegenerate,
the map $\bT$ is an isomorphism of $\CC_v(\!(t)\!)$-vector spaces.

For $a\in H_v$, let $M(a)\in\Mat_d(\ov{K}(t))$ denote the matrix representing
multiplication by $a$ with respect to the ordered basis $\{1,T,\ldots,T^{d-1}\}$. More
precisely, we have
\[
    aT^j = \sum_{i=0}^{d-1}M(a)_{ij}T^i
    \qquad (0\leq j\leq d-1).
\]
The following lemma gives the basic properties of multiplication matrices and trace
coordinates.

\begin{lem}\label{lem:trace-coordinates}
For every $a\in H_v$ and $F\in\HH_v$, we have
\[
    \bT(aF)=M(a)^{\mathsf T}\bT(F)
    \qquad\text{and}\qquad \bT(F^{(-1)})=\bT(F)^{(-1)}.
\]
Moreover, we have
\[
    \det M(a)=\Nr_{H_v/\ov{K}(t)}(a)
\]
where $\Nr_{H_v/\ov{K}(t)}$ denotes the field norm.
\end{lem}

\begin{proof}
For $0\leq j\leq d-1$, we have
\[
    \begin{aligned}
        \Tr(T^jaF) =\Tr((aT^j)F)=\sum_{i=0}^{d-1}M(a)_{ij}\Tr(T^iF).
    \end{aligned}
\]
This gives $\bT(aF)=M(a)^{\mathsf T}\bT(F)$. On the other hand, since Frobenius twisting
fixes $T$ and commutes with the field trace, we also have $\bT(F^{(-1)})=\bT(F)^{(-1)}$.
Finally, the determinant identity follows from the definition of the field norm.
\end{proof}

To establish the $v$-adic entireness of the trace coordinates of
$\mathfrak{L}_{\bs,\bu}^{[v]}(T,t)$, we first consider the traces of powers of $T$.

\begin{lem}\label{lem:power-trace}
For $m\geq0$, put $R_m(t):=\Tr(T^m)$. Then $R_m(t)\in\FF_q[t^{-1}]$, with pole order at
$t=0$ at most $\lfloor m/d\rfloor$. In particular, $R_m(t)\in\FF_q$ for $0\leq m<d$.
\end{lem}

\begin{proof}
Since $R_0=d\in\FF_q$, Newton's identities applied to $v(X)-t^{-1}$ show that
$R_m\in\FF_q$ for $0\leq m<d$, as $t^{-1}$ occurs only in the constant term. For
$m\geq d$, the relation \eqref{eq:Hv-relation} gives
\[
    R_m = -a_{d-1}R_{m-1}-\cdots-a_0R_{m-d} +t^{-1}R_{m-d}.
\]
The assertion follows by induction on $m$.
\end{proof}

We then show that all coefficients of $\mathfrak{L}_{\bs,\bu}^{[v]}(T,t)$ with respect
to the basis $\{1,T,\ldots,T^{d-1}\}$ are $v$-adic entire, and so are its trace
coordinates.
\begin{prop}\label{prop:trace-entire}
Let $\bs=(s_1,\ldots,s_r)\in\Ical$ be non-empty and
$\bu=(u_1,\ldots,u_r)\in (\ov{K}^{\times})^r\cap\DD_{\bs,v}^{\Conv}$. For every
$0\leq m\leq d-1$, write
\[
    T^m\mathfrak{L}_{\bs,\bu}^{[v]}(T,t) = \sum_{j=0}^{d-1}f_{m,j}(t)T^j.
\]
Then we have
\[
    f_{m,j}(t)\in\Ecal_v
    \qquad (0\leq j\leq d-1)\qquad \text{and}\qquad \Tr\left(T^m\mathfrak{L}_{\bs,\bu}^{[v]}(T,t)
    \right)
    \in\Ecal_v.
\]
Consequently,
\[
    \bT\left(\mathfrak{L}_{\bs,\bu}^{[v]}(T,t)
    \right)
    \in\Mat_{d\times1}(\Ecal_v).
\]
\end{prop}

\begin{proof}
Fix $0\leq m\leq d-1$. Consider a summand of \eqref{eq:L-series} indexed by
$i_1>\cdots>i_r\geq0$, and put $n:=i_1s_1+\cdots+i_rs_r$ and
$\mathbf{i}=(i_1,\ldots,i_r)$. Since
\[
    \deg_T\left(T^m\prod_{\ell=1}^r \Dcal_{v,i_\ell}(T)^{s_\ell}
    \right)
    \leq m+(d-1)n,
\]
the reduction argument in the proof of \cref{lem:L-in-Kv} gives
\[
    t^nT^m \prod_{\ell=1}^r \Dcal_{v,i_\ell}(T)^{s_\ell} = \sum_{j=0}^{d-1} P_{m,j,\mathbf i}(t)T^j,
\]
where $P_{m,j,\mathbf i}(t)\in A[t]$ has degree at most $n$ and satisfies
\[
    \ord_t P_{m,j,\mathbf i}(t) \geq \left\lceil\frac{n-m}{d}\right\rceil.
\]
Since $m<d$, this lower bound is nonnegative and tends to infinity as $i_1\to\infty$.
Thus the defining series of each $f_{m,j}(t)$ converges $t$-adically and has no negative
powers of $t$.

We next prove that $f_{m,j}(t)$ is $v$-adic entire. Fix $R>0$ and put
\[
    C_R:= \prod_{k\geq1} \max\{1,|v|_v^{q^k}R\}.
\]
By \eqref{eq:v-adic AT}, we have $|\Omega_v^{(i)}(t)|_v\leq C_R$ for all $i\geq0$ and
$|t|_v\leq R$. Moreover, since $P_{m,j,\mathbf i}(t)\in A[t]$ has degree at most $n$, we
have
\[
    |P_{m,j,\mathbf i}(t)|_v \leq\max\{1,R\}^n
    \qquad (|t|_v\leq R).
\]
Consequently, the summand of $f_{m,j}(t)$ corresponding to $\mathbf i=(i_1,\ldots,i_r)$
is bounded by
\[
    C_R^w |u_1|_v^{q^{i_1}} \max\{1,R\}^{wi_1},
\]
where we used $n\leq wi_1$ and $|u_\ell|_v\leq1$ for $\ell\geq2$. Since $|u_1|_v<1$,
this bound tends to zero as $i_1\to\infty$. Since, for each fixed $i_1$, there are only
finitely many tuples $\mathbf{i}$ with $i_1>\cdots>i_r\geq 0$, the defining series
converges uniformly on every closed disk. Hence $f_{m,j}(t)$ is entire.

Finally, by \eqref{eq:v-adic AT}, there exists a finite extension $L/K_v$ containing all
the coefficients of $\Omega_v$ and $u_1,\ldots,u_r$. Recall that
\[
    f_{m,j}(t) = \sum_{i_1>\cdots>i_r\geq0} P_{m,j,\mathbf i}(t) \prod_{\ell=1}^r
    (\Omega_v^{s_\ell}u_\ell)^{(i_\ell)}.
\]
Since $P_{m,j,\mathbf i}(t)\in A[t]$, all the coefficients appearing in this sum belong
to $L\cap\ov{K}$. Moreover, since $\ord_t P_{m,j,\mathbf i}(t)\to\infty$ as
$i_1\to\infty$, only finitely many terms contribute to each coefficient of $f_{m,j}(t)$.
Thus all the coefficients of $f_{m,j}(t)$ belong to $L\cap\ov{K}$. Together with the
entireness established above, we obtain $f_{m,j}(t)\in\Ecal_v$ and hence, by
\cref{lem:power-trace},
\[
    \Tr\left(T^m\mathfrak{L}_{\bs,\bu}^{[v]}(T,t)
    \right)
    = \sum_{j=0}^{d-1}f_{m,j}(t)R_j(t) \in\Ecal_v,
\]
which completes the proof.
\end{proof}

\subsection{Frobenius difference equations}
\label{subsec:Frobenius-equations}
We now use trace coordinates to convert the Frobenius difference equation in
\cref{prop:L-Frobenius} into a Frobenius system to which \cref{thm:vadic-ABP} can be
applied. Recall that
\[
    t\delta_v(T) = \frac{1-vt}{T-\theta} \in H_v.
\]
Put $M_v(t):=M(t\delta_v(T))$. Then we have the following properties of $M_v(t)$.

\begin{lem}\label{lem:Mv}
We have $M_v(t)\in\Mat_d(K[t])$ and
\begin{equation}\label{eq:det-Mv}
    \det M_v(t) = (-1)^{d+1}t(1-vt)^{d-1}.
\end{equation}
\end{lem}
\begin{proof}
By \eqref{eq:Hv-relation}, we have
\[
    tT^d = 1-t(a_{d-1}T^{d-1}+\cdots+a_0).
\]
Since $\deg_T\delta_v(T)=d-1$, repeated use of this identity expresses $t\delta_v(T)T^j$
in the basis $\{1,T,\ldots,T^{d-1}\}$ with coefficients in $K[t]$ for every
$0\leq j\leq d-1$. Hence $M_v(t)\in\Mat_d(K[t])$.

On the other hand, since $v(X)-t^{-1}$ is the minimal polynomial of $T$ over
$\ov{K}(t)$, it follows from \cref{lem:trace-coordinates} that
\[
    \det M_v(t) =\Nr_{H_v/\ov{K}(t)}(t\delta_v(T)) =\frac{(1-vt)^d}{(-1)^{d+1}(t^{-1}-v)}
    =(-1)^{d+1}t(1-vt)^{d-1}.
\]
\end{proof}

Let $\bs=(s_1,\ldots,s_r)\in\Ical$ be non-empty and
$\bu\in(\ov{K}^{\times})^r\cap\DD_{\bs,v}^{\Conv}$. For $1\leq\ell\leq r$, put
\[
    \bs_\ell:=(s_1,\ldots,s_\ell),
    \qquad \bu_\ell:=(u_1,\ldots,u_\ell),
    \qquad w_\ell:=s_1+\cdots+s_\ell.
\]
Also, put $w_0:=0$ and $w:=w_r$. We define
\begin{equation}\label{eq:psi-su}
    \psi_{\bs,\bu}(t) :=
    \begin{pmatrix}
        \Omega_v^w\\[2mm]
        \bT\left(\Omega_v^{w-w_1} \mathfrak{L}_{\bs_1,\bu_1}^{[v]}(T,t)\right)\\
        \vdots\\
        \bT\left(\Omega_v^{w-w_r} \mathfrak{L}_{\bs_r,\bu_r}^{[v]}(T,t)
        \right)
    \end{pmatrix}=\begin{pmatrix}
        \Omega_v^w\\[2mm]
        \Omega_v^{w-w_1}\bT(\mathfrak{L}_{\bs_1,\bu_1}^{[v]}(T,t))\\
        \vdots\\
        \Omega_v^{w-w_r}\bT( \mathfrak{L}_{\bs_r,\bu_r}^{[v]}(T,t))
    \end{pmatrix}.
\end{equation}
By \eqref{eq:v-adic AT} and \cref{prop:trace-entire}, we have
\[
    \psi_{\bs,\bu}(t) \in\Mat_{(1+rd)\times1}(\Ecal_v).
\]

\begin{prop}\label{prop:Frobenius-system}
Let the notation be as above. We have
\begin{equation}\label{eq:Frobenius-system}
    \psi_{\bs,\bu}^{(-1)} = \Phi_{\bs,\bu}\psi_{\bs,\bu},
\end{equation}
where $\Phi_{\bs,\bu}\in\Mat_{1+rd}(\ov{K}[t])$ is the block lower triangular matrix
\[
    \Phi_{\bs,\bu}(t) =
    \resizebox{0.8\linewidth}{!}{$
    \left(
    \begin{array}{c|c|c|c|c}
        (1-vt)^w & 0 & 0 & \cdots & 0 \\ \hline
        u_1^{q^{-1}}(1-vt)^w\bT(1) &
        \begin{gathered}
            (1-vt)^{w-w_1}\\
            (M_v(t)^{\mathsf T})^{w_1}
        \end{gathered}
        & 0 & \cdots & 0 \\ \hline
        0 &
        \begin{gathered}
            u_2^{q^{-1}}(1-vt)^{w-w_1}\\
            (M_v(t)^{\mathsf T})^{w_1}
        \end{gathered}
        &
        \begin{gathered}
            (1-vt)^{w-w_2}\\
            (M_v(t)^{\mathsf T})^{w_2}
        \end{gathered}
        & \ddots & \vdots \\ \hline
        \vdots & \ddots & \ddots & \ddots & 0 \\ \hline
        0 & \cdots & 0 &
        \begin{gathered}
            u_r^{q^{-1}}(1-vt)^{w-w_{r-1}}\\
            (M_v(t)^{\mathsf T})^{w_{r-1}}
        \end{gathered}
        & (M_v(t)^{\mathsf T})^w
    \end{array}
    \right)
    $}.
\]
Here, the first diagonal block is of size $1\times1$, the first subdiagonal block is of
size $d\times1$, and all the remaining displayed blocks are of size $d\times d$.
\end{prop}

\begin{proof}
By \eqref{eq:v-adic AT}, the first coordinate satisfies
\[
    (\Omega_v^w)^{(-1)} = (1-vt)^w\Omega_v^w.
\]

For $2\leq\ell\leq r$, applying the trace coordinate map to \cref{prop:L-Frobenius} and
using \eqref{eq:v-adic AT} and \cref{lem:trace-coordinates}, we obtain
\[
    \begin{aligned}
        & \bT\left(\Omega_v^{w-w_\ell} \mathfrak{L}_{\bs_\ell,\bu_\ell}^{[v]}(T,t)\right)^{(-1)} \\
        &\quad= u_\ell^{q^{-1}} (1-vt)^{w-w_{\ell-1}} \left(M_v(t)^{\mathsf T}\right)^{w_{\ell-1}}
        \bT\left(\Omega_v^{w-w_{\ell-1}} \mathfrak{L}_{\bs_{\ell-1},\bu_{\ell-1}}^{[v]}(T,t)
        \right)
        \\
        &\qquad\quad+ (1-vt)^{w-w_\ell} \left(M_v(t)^{\mathsf T}\right)^{w_\ell}
        \bT\left(\Omega_v^{w-w_\ell} \mathfrak{L}_{\bs_\ell,\bu_\ell}^{[v]}(T,t)\right).
    \end{aligned}
\]
For $\ell=1$, the same argument, together with
$\mathfrak{L}_{\varnothing,\varnothing}^{[v]}(T,t)=1$, gives
\[
    \begin{aligned}
        & \bT\left(\Omega_v^{w-w_1} \mathfrak{L}_{\bs_1,\bu_1}^{[v]}(T,t)\right)^{(-1)} \\
        &\quad= u_1^{q^{-1}}(1-vt)^w \bT(1)\Omega_v^w + (1-vt)^{w-w_1}
        \left(M_v(t)^{\mathsf T}\right)^{w_1} \bT\left(\Omega_v^{w-w_1}
        \mathfrak{L}_{\bs_1,\bu_1}^{[v]}(T,t)\right).
    \end{aligned}
\]
These identities give \eqref{eq:Frobenius-system}. Moreover, by
\cref{lem:Mv,lem:power-trace}, we have $M_v(t)\in\Mat_d(K[t])$ and
$\bT(1)\in\Mat_{d\times1}(\FF_q)$. Hence $\Phi_{\bs,\bu}(t)$ has entries in $\ov{K}[t]$.
\end{proof}

\begin{cor}\label{cor:ABP-admissible}
Let the notation be as above. The pair $\left(\Phi_{\bs,\bu},\psi_{\bs,\bu}\right)$
satisfies the hypotheses in \cref{thm:vadic-ABP} at $\gamma=v^{-1}$.
\end{cor}

\begin{proof}
By construction, we have
\[
    \Phi_{\bs,\bu}\in\Mat_{1+rd}(\ov{K}[t]),
    \qquad \psi_{\bs,\bu}\in\Mat_{(1+rd)\times1}(\Ecal_v).
\]
Since $\Phi_{\bs,\bu}$ is block lower triangular, \cref{lem:Mv} shows that every zero of
$\det\Phi_{\bs,\bu}(t)$ belongs to $\{0,v^{-1}\}$. Moreover, $v^{-1}\notin\ov{\FF}_q$
and $(v^{-1})^{(-i)}\notin\{0,v^{-1}\}$ for every $i\geq1$. Thus all the hypotheses of
\cref{thm:vadic-ABP} are satisfied.
\end{proof}

\subsection{Specialization}
\label{subsec:frobenius-fibers}
In this subsection, we study the specializations of $\mathfrak L_{\bs,\bu}^{[v]}(T,t)$
and its trace coordinates at $t=v^{-q^N}$ for $N\geq 0$.

We begin with the following useful formula.

\begin{lem}\label{lem:trace-specialization}
Let $a\in\ov{K}^\times$ and suppose that $v(T)-a^{-1}$ is separable. Let
\[
    F(T,t)=\sum_{j=0}^{d-1}f_j(t)T^j\in\HH_v
\]
be such that each $f_j(t)$ is defined at $t=a$. Then
\[
    \Tr(F(T,t))(a) = \sum_{v(\beta)=a^{-1}}F(\beta,a),
\]
where the sum runs over all roots of $v(T)-a^{-1}$.
\end{lem}
\begin{proof}
For $0\leq j\leq d-1$, we have
\[
    R_j(t)=\Tr(T^j)=\tr(M(T)^j).
\]
Since $a\neq0$, the matrix $M(T)(a)$ represents multiplication by $T$ on
$\ov{K}[T]/(v(T)-a^{-1})$. As $v(T)-a^{-1}$ is separable, its roots are precisely the
eigenvalues of $M(T)(a)$. Consequently,
\[
    R_j(a) = \sum_{v(\beta)=a^{-1}}\beta^j.
\]
It follows that
\[
    \begin{aligned}
        \Tr(F(T,t))(a) &=\sum_{j=0}^{d-1}f_j(a)R_j(a)=\sum_{v(\beta)=a^{-1}}F(\beta,a),
    \end{aligned}
\]
as desired.
\end{proof}

We next establish two identities needed to compute the specializations of
$\mathfrak{L}_{\bs,\bu}^{[v]}(T,t)$ at $t=v^{-q^N}$.

\begin{lem}\label{lem:Omega-shift}
For $N\geq1$ and $h\geq0$, we have
\begin{equation}\label{eq:Omega-shift}
    \left.\Omega_v^{(N+h)}(t)t^{N+h}\right|_{t=v^{-q^N}} = v^{-Nq^N}
    \left(\left.\Omega_v^{(h)}(t)t^h\right|_{t=v^{-1}}\right)^{q^N}.
\end{equation}
Moreover, if $i<N$, then
\begin{equation}\label{eq:Omega-vanishing}
    \Omega_v^{(i)}(v^{-q^N})=0.
\end{equation}
\end{lem}

\begin{proof}
Recall that $\Omega_v^{(i)}(t)=\prod_{j>i}(1-v^{q^j}t)$. If $i<N$, the factor with $j=N$
vanishes at $t=v^{-q^N}$, which proves \eqref{eq:Omega-vanishing}. On the other hand,
for $i=N+h$, we have
\[
    \begin{aligned}
        \Omega_v^{(N+h)}(v^{-q^N})v^{-(N+h)q^N} &= v^{-(N+h)q^N} \prod_{j>N+h}(1-v^{q^j-q^N})\\
        &= v^{-Nq^N} \left(v^{-h}\prod_{j>h}(1-v^{q^j-1})\right)^{q^N}
    \end{aligned}
\]
This proves \eqref{eq:Omega-shift}.
\end{proof}

For each $N\geq 1$, we put
\[
    \lambda_{v,N} := v^{-Nq^N} \prod_{j=1}^N \Delta_v(\theta^{q^N},\theta^{q^j}) \in\ov{K}^{\times}.
\]
Note that the product is nonzero since $\Delta_v(\theta^{q^N},\theta^{q^j})\neq0$ for
$1\leq j<N$ and $\Delta_v(\theta^{q^N},\theta^{q^N}) =v'(\theta^{q^N})\neq0$.

\begin{lem}\label{lem:D-shift}
For $N\geq1$ and $h\geq0$, we have
\begin{equation}\label{eq:D-shift}
    \Dcal_{v,N+h}(\theta^{q^N}) = v^{Nq^N}\lambda_{v,N} \Dcal_{v,h}(\theta)^{q^N}.
\end{equation}
If $\beta\neq\theta^{q^N}$ is another root of $v(T)-v^{q^N}$, then
\begin{equation}\label{eq:D-vanishing}
    \Dcal_{v,i}(\beta)=0
    \qquad(i\geq N).
\end{equation}
\end{lem}
\begin{proof}
By \eqref{eq:Delta-twist}, we have
\[
    \begin{aligned}
        \Dcal_{v,N+h}(\theta^{q^N}) &= \left(\prod_{j=1}^N \Delta_v(\theta^{q^N},\theta^{q^j})
        \right)
        \left(\prod_{j=1}^h \Delta_v(\theta,\theta^{q^j})\right)^{q^N}\\
        &= v^{Nq^N}\lambda_{v,N} \Dcal_{v,h}(\theta)^{q^N},
    \end{aligned}
\]
which proves \eqref{eq:D-shift}. If $\beta\neq\theta^{q^N}$ is another root of
$v(T)-v^{q^N}$, then $\Delta_v(\beta,\theta^{q^N})=0$ and \eqref{eq:D-vanishing} follows
immediately from the definition of $\Dcal_{v,i}(\beta)$ for every $i\geq N$.
\end{proof}

Combining the preceding three lemmas, we determine the specialization of
$\mathfrak L_{\bs,\bu}^{[v]}(T,t)$ on the fiber $v(T)=v^{q^N}$.

\begin{prop}\label{prop:fiber-specialization}
Let $\bs=(s_1,\ldots,s_r)\in\Ical$ be non-empty and let
$\bu=(u_1,\ldots,u_r)\in (\ov{K}^{\times})^r\cap\DD_{\bs,v}^{\Conv}$. Let $w=\wt(\bs)$
and $N\geq1$. For every root $\beta$ of $v(T)-v^{q^N}$, we have
\begin{equation}\label{eq:fiber-specialization}
    \mathfrak L_{\bs,\bu}^{[v]}(\beta,v^{-q^N}) =
    \begin{cases}
        \displaystyle \lambda_{v,N}^w \left(\widetilde{\pi}_v^w\Li_{\bs}(\bu)_v\right)^{q^N}, &
        \beta=\theta^{q^N}, \\[3mm]
        0, & \beta\neq\theta^{q^N}.
    \end{cases}
\end{equation}
Consequently,
\begin{equation}\label{eq:trace-fiber-specialization}
    \bT\left(\mathfrak L_{\bs,\bu}^{[v]}(T,t)\right)(v^{-q^N}) = \lambda_{v,N}^w
    \left(\widetilde{\pi}_v^w\Li_{\bs}(\bu)_v\right)^{q^N} \mathbf V_N,
\end{equation}
where
\[
    \mathbf V_N :=
    \begin{pmatrix}
        1\\
        \theta^{q^N}\\
        \vdots\\
        (\theta^{q^N})^{d-1}
    \end{pmatrix}.
\]
\end{prop}
\begin{proof}
By the proof of \cref{prop:trace-entire}, the series \eqref{eq:L-series} can be
evaluated termwise at $(T,t)=(\beta,v^{-q^N})$. By \cref{lem:Omega-shift}, every summand
with $i_\ell<N$ for some $\ell$ vanishes. Thus only indices of the form
\[
    i_\ell=N+h_\ell,
    \qquad h_1>\cdots>h_r\geq0,
\]
contribute.

For $\beta=\theta^{q^N}$, applying \cref{lem:Omega-shift,lem:D-shift} to each factor and
using \cref{prop:L-specialization}, we obtain
\[
    \begin{aligned}
        \mathfrak L_{\bs,\bu}^{[v]} (\theta^{q^N},v^{-q^N}) = \lambda_{v,N}^w \left(\mathfrak
        L_{\bs,\bu}^{[v]}(\theta,v^{-1})\right)^{q^N}= \lambda_{v,N}^w
        \left(\widetilde{\pi}_v^w\Li_{\bs}(\bu)_v\right)^{q^N}.
    \end{aligned}
\]
If $\beta\neq\theta^{q^N}$ is another root of $v(T)-v^{q^N}$, then
$\Dcal_{v,i}(\beta)=0$ for every $i\geq N$ by \cref{lem:D-shift}. Hence all remaining
summands vanish at $(T,t)=(\beta,v^{-q^N})$. This proves
\eqref{eq:fiber-specialization}.

Since $v'(T)\in\FF_q[T]$ is nonzero and $v^{q^N}$ is transcendental over $\FF_q$, the
polynomial $v(T)-v^{q^N}$ is separable. Therefore, by
\cref{prop:trace-entire,lem:trace-specialization}, for $0\leq m\leq d-1$, we have
\[
    \begin{aligned}
        \Tr\left(T^m\mathfrak L_{\bs,\bu}^{[v]}(T,t)\right)(v^{-q^N}) = \sum_{v(\beta)=v^{q^N}}
        \beta^m \mathfrak L_{\bs,\bu}^{[v]}(\beta,v^{-q^N}) = (\theta^{q^N})^m \lambda_{v,N}^w
        \left(\widetilde{\pi}_v^w\Li_{\bs}(\bu)_v\right)^{q^N}.
    \end{aligned}
\]
This gives \eqref{eq:trace-fiber-specialization} and completes the proof.
\end{proof}

To recover the specialization at $(T,t)=(\theta,v^{-1})$ from the trace coordinates,
write
\begin{equation}\label{eq:eta-v}
    \frac{\Delta_v(T,\theta)}{v'(\theta)} = \sum_{m=0}^{d-1}\eta_{v,m}T^m,
    \qquad \boldsymbol{\eta}_v := (\eta_{v,0},\ldots,\eta_{v,d-1}).
\end{equation}

\begin{lem}\label{lem:evaluation-row}
Let $\bs=(s_1,\ldots,s_r)\in\Ical$ be non-empty and
$\bu\in(\ov{K}^{\times})^r\cap\DD_{\bs,v}^{\Conv}$. Then
\[
    \boldsymbol{\eta}_v\, \bT\left(\mathfrak L_{\bs,\bu}^{[v]}(T,t)\right)(v^{-1}) =
    \widetilde{\pi}_v^w\Li_{\bs}(\bu)_v.
\]
\end{lem}

\begin{proof}
Since $v(T)-v$ is separable, \cref{prop:trace-entire,lem:trace-specialization} and
\eqref{eq:eta-v} give
\[
    \boldsymbol{\eta}_v\, \bT\left(\mathfrak L_{\bs,\bu}^{[v]}(T,t)\right)(v^{-1}) =
    \sum_{v(\beta)=v} \frac{\Delta_v(\beta,\theta)}{v'(\theta)} \mathfrak
    L_{\bs,\bu}^{[v]}(\beta,v^{-1})= \mathfrak L_{\bs,\bu}^{[v]}(\theta,v^{-1}),
\]
where the second equality follows from $\Delta_v(\theta,\theta)=v'(\theta)$ and
$\Delta_v(\beta,\theta)=0$ for every other root $\beta$ of $v(T)-v$. The result follows
from \cref{prop:L-specialization}.
\end{proof}

\appendix
\section{Proofs of the Main Results}
\label{sec:proof-main}

In this appendix, we complete the proofs of
\cref{thm:main,thm:mainMZV,cor:transcendence}. Our arguments follow the strategy of
Chang \cite[\S4]{Chang14} and its $v$-adic adaptation by Chen \cite[\S3.1]{ChenABP}. In
\S\ref{sec:A1}, we will apply the $v$-adic Anderson--Brownawell--Papanikolas criterion
(see Theorem \ref{thm:vadic-ABP}) to the Frobenius systems constructed in the previous
section and use the specializations of their trace coordinates to prove the desired
linear independence results for $v$-adic Carlitz multiple polylogarithms. In
\S\ref{sec:A2}, we deduce the corresponding results for $v$-adic multiple zeta values
and prove Corollary~\ref{cor:transcendence}.

\subsection{\texorpdfstring{Proof of Theorem~\ref{thm:main}}{Proof of Theorem 1.2}}\label{sec:A1}
We first establish two auxiliary lemmas.

\begin{lem}\label{lem:orbit-nondegeneracy}
Let
\[
    Q(t)=(Q_0(t),\ldots,Q_{d-1}(t)) \in\Mat_{1\times d}(\ov{K}[t])
\]
be nonzero. Then
\[
    Q(v^{-q^N})\mathbf V_N\neq0
\]
for all sufficiently large $N\geq1$.
\end{lem}

\begin{proof}
Suppose that $Q(v^{-q^N})\mathbf V_N=0$ for infinitely many $N$. Consider the rational
function
\[
    R(T):= \sum_{m=0}^{d-1} Q_m(v(T)^{-1})T^m \in\ov{K}(T).
\]
Since $v(\theta^{q^N})=v^{q^N}$, we have
\[
    R(\theta^{q^N}) = Q(v^{-q^N})\mathbf V_N = 0
\]
for infinitely many $N$. Since $\theta^{q^N}$ are all distinct, we obtain $R(T)=0$. By
the linear independence of $\{1,T,\ldots,T^{d-1}\}$ over $\ov{K}(t)$, we have $Q_m(t)=0$
for every $0\leq m\leq d-1$, which is a contradiction.
\end{proof}

\begin{lem}\label{lem:Frobenius-rigidity}
Let $S(T)\in\ov{K}(T)$ and $c\in\ov{K}$. Suppose that
\[
    S(\theta^{q^N})=c^{q^N}
\]
for all sufficiently large $N$. Then $S(T)\in\FF_q(T)$ and $S(\theta)=c\in K$.
\end{lem}

\begin{proof}
For all sufficiently large $N$, we have
\[
    S^{(1)}(\theta^{q^{N+1}}) = S(\theta^{q^N})^q = c^{q^{N+1}} = S(\theta^{q^{N+1}}).
\]
Since $\theta^{q^N}$ are all distinct, we obtain $S^{(1)}=S$. Write $S=B/C$, where
$B,C\in\ov{K}[T]$ are relatively prime and $C$ is monic. Then
\[
    B^{(1)}C=BC^{(1)}.
\]
Since $B$ and $C$ are relatively prime, we have $C\mid C^{(1)}$. As $C$ and $C^{(1)}$
are monic of the same degree, we obtain $C^{(1)}=C$, and hence $B^{(1)}=B$. Therefore,
$B,C\in\FF_q[T]$, so $S\in\FF_q(T)$. In particular, $S$ is regular at $T=\theta$, and
\[
    S(\theta)^{q^N}=S(\theta^{q^N}) =c^{q^N}
\]
for all sufficiently large $N$. Thus, we get $S(\theta)=c\in K$.
\end{proof}

We now study linear relations among $v$-adic Carlitz multiple polylogarithms of the same
weight.

\begin{lem}\label{lem:fixed-weight}
Let $w\geq1$ and let
\[
    Z_j:=\Li_{\bs_j}(\bu_j)_v\qquad (1\leq j\leq m),
\]
where each $\bs_j\in\Ical$ is a non-empty index of weight $w$ and
$\bu_j\in (\ov{K}^{\times})^{\dep(\bs_j)} \cap\DD_{\bs_j,v}^{\Conv}$. If
$Z_1,\ldots,Z_m$ are linearly independent over $K$, then they are linearly independent
over $\ov{K}$.
\end{lem}

\begin{proof}
Suppose that $Z_1,\ldots,Z_m$ are linearly dependent over $\ov{K}$. By choosing a
minimally dependent subset and relabeling if necessary, we may assume that
\begin{equation}\label{eq:minimal-relation}
    Z_1+c_2Z_2+\cdots+c_mZ_m=0,
    \qquad c_2,\ldots,c_m\in\ov{K},
\end{equation}
where $Z_2,\ldots,Z_m$ are linearly independent over $\ov{K}$. Put $c_1:=1$.

For each $1\leq j\leq m$, let
\[
    \psi_j^{(-1)}=\Phi_j\psi_j
\]
be the Frobenius system constructed in \cref{prop:Frobenius-system} corresponding to
$Z_j$. Consider the block system
\[
    \psi:=
    \begin{pmatrix}
        \psi_1\\
        \vdots\\
        \psi_m
    \end{pmatrix},
    \qquad \Phi:=\operatorname{diag}(\Phi_1,\ldots,\Phi_m).
\]
Then $\psi^{(-1)}=\Phi\psi$. By \cref{cor:ABP-admissible}, the pair $(\Phi,\psi)$
satisfies the hypotheses of \cref{thm:vadic-ABP} at $t=v^{-1}$. For each $j$, put
$r_j:=\dep(\bs_j)$ and define
\[
    \rho_j:=
    \begin{pNiceMatrix}
        0,\ldots,0 & \!\!\!\mid\!\!\! & c_j\boldsymbol{\eta}_v
        \CodeAfter
        \UnderBrace[shorten,yshift=3pt]
            {1-1}{1-1}{\scriptscriptstyle 1+(r_j-1)d}
        \UnderBrace[shorten,yshift=3pt]
            {1-3}{1-3}{\scriptscriptstyle d}
    \end{pNiceMatrix}
    \in\Mat_{1\times(1+r_jd)}(\ov{K})
    \qquad\text{and}\qquad \rho:= \left(\rho_1\mid\cdots\mid\rho_m\right).
\]
By \cref{lem:evaluation-row}, we have
\[
    \rho\psi(v^{-1}) = \widetilde{\pi}_v^w \sum_{j=1}^m c_jZ_j = 0.
\]
Hence, by \cref{thm:vadic-ABP}, there exists
\[
    \mathbf p(t)= \left(\mathbf p_1(t)\mid\cdots\mid\mathbf p_m(t)\right),
\]
with entries in $\ov{K}[t]$ such that
\begin{equation}\label{eq:fixed-lift}
    \mathbf p(v^{-1})=\rho,
    \qquad \mathbf p(t)\psi(t)=0
\end{equation}
where $\mathbf p_j(t)$ is the block corresponding to $\psi_j$. Let
$Q_j(t)\in\Mat_{1\times d}(\ov{K}[t])$ denote the last $d$ entries of $\mathbf p_j(t)$.
By \eqref{eq:fixed-lift}, we have
\begin{equation}\label{eq:Q-initial}
    Q_j(v^{-1})=c_j\boldsymbol{\eta}_v.
\end{equation}
In particular, $Q_1(v^{-1})=\boldsymbol{\eta}_v\neq0$, so $Q_1(t)\neq0$.

For every $N\geq1$, put
\[
    a_{j,N}:=Q_j(v^{-q^N})\mathbf V_N.
\]
By \cref{lem:orbit-nondegeneracy}, we have $a_{1,N}\neq0$ for all sufficiently large
$N$. On the other hand, by \eqref{eq:Omega-vanishing} and
\cref{prop:fiber-specialization}, we have
\[
    \psi_j(v^{-q^N}) =
    \begin{pNiceMatrix}
        0\\
        \vdots\\
        0\\[3pt]
        \lambda_{v,N}^w
        (\widetilde{\pi}_v^w Z_j)^{q^N}\mathbf V_N
        \CodeAfter
        \SubMatrix{.}{1-1}{3-1}{\rbrace}
            [right-xshift=3mm]
        \SubMatrix{.}{4-1}{4-1}{\rbrace}
            [right-xshift=3mm]
        \tikz \node[anchor=west,inner sep=0pt]
            at ([xshift=7mm]$(1-|2)!0.5!(4-|2)$)
            {$\scriptscriptstyle 1+(r_j-1)d$};
        \tikz \node[anchor=west,inner sep=0pt]
            at ([xshift=7mm]$(4-|2)!0.5!(5-|2)$)
            {$\scriptscriptstyle d$};
    \end{pNiceMatrix}\hspace{17mm}
\]
for each $1\leq j\leq m$. Specializing the second identity in \eqref{eq:fixed-lift} at
$t=v^{-q^N}$, we obtain
\[
    \lambda_{v,N}^w \left(\widetilde{\pi}_v^w\right)^{q^N} \sum_{j=1}^m a_{j,N}Z_j^{q^N} =0.
\]
Since $\lambda_{v,N}\neq0$ and $\widetilde{\pi}_v\neq0$, it follows that
\[
    \sum_{j=1}^m a_{j,N}^{q^{-N}}Z_j=0.
\]
Since $Z_2,\ldots,Z_m$ are linearly independent over $\ov{K}$, the space of
$\ov{K}$-linear relations among $Z_1,\ldots,Z_m$ is one-dimensional. Comparing this
relation with \eqref{eq:minimal-relation}, we obtain
\begin{equation}\label{eq:fixed-ratios}
    \frac{a_{j,N}}{a_{1,N}} = c_j^{q^N}
    \qquad(1\leq j\leq m)
\end{equation}
for all sufficiently large $N$.

For each $1\leq j\leq m$, define
\[
    A_j(T):= Q_j(v(T)^{-1})
    \begin{pmatrix}
        1\\
        T\\
        \vdots\\
        T^{d-1}
    \end{pmatrix}
    \in\ov{K}(T).
\]
Since $Q_1(t)\neq0$ and $\{1,T,\ldots,T^{d-1}\}$ is linearly independent over
$\ov{K}(t)$, we have $A_1(T)\neq0$. Moreover,
\[
    A_j(\theta^{q^N}) = Q_j(v^{-q^N})\mathbf V_N = a_{j,N}.
\]
Therefore, by \eqref{eq:fixed-ratios}, the rational function
\[
    S_j(T):=\frac{A_j(T)}{A_1(T)}
\]
satisfies
\[
    S_j(\theta^{q^N})=c_j^{q^N}
\]
for all sufficiently large $N$. By \cref{lem:Frobenius-rigidity}, we obtain
$c_j=S_j(\theta)\in K$ for every $1\leq j\leq m$. Thus \eqref{eq:minimal-relation} gives
a nontrivial $K$-linear relation among $Z_1,\ldots,Z_m$, which is a contradiction.
\end{proof}

We next consider linear relations involving values of distinct weights.

\begin{lem}\label{lem:cross-weight}
Let $w_1>\cdots>w_s\geq 1$ be distinct positive integers. For each $1\leq i\leq s$, let
\[
    Z_{ij}:=\Li_{\bs_{ij}}(\bu_{ij})_v
    \qquad(1\leq j\leq m_i),
\]
where $\bs_{ij}$ is a non-empty index of weight $w_i$ and $\bu_{ij}\in (\ov{K}^{\times})^{\dep(\bs_{ij})} \cap\DD_{\bs_{ij},v}^{\Conv}$. If
$Z_{i1},\ldots,Z_{im_i}$ are linearly independent over $K$ for every $1\leq i\leq s$,
then
\[
    \{1\}\cup \{Z_{ij}\mid 1\leq i\leq s,\ 1\leq j\leq m_i\}
\]
is linearly independent over $\ov{K}$.
\end{lem}

\begin{proof}
By \cref{lem:fixed-weight}, the values $Z_{i1},\ldots,Z_{im_i}$ are linearly independent
over $\ov{K}$ for each $i$. Suppose, on the contrary, that
\[
    c_0+\sum_{i=1}^s\sum_{j=1}^{m_i}c_{ij}Z_{ij}=0,
    \qquad c_0,c_{ij}\in\ov{K},
\]
is a nontrivial relation. After discarding the weights that do not occur and relabeling if necessary, we may assume that $c_{1j}\neq0$ for some $1\leq j\leq m_1$.

For each $Z_{ij}$, let
\[
    \psi_{ij}^{(-1)}=\Phi_{ij}\psi_{ij}
\]
be the Frobenius system constructed in \cref{prop:Frobenius-system}. We put
\[
    \widehat{\psi}_{ij} :=\Omega_v^{w_1-w_i}\psi_{ij},
    \qquad \widehat{\Phi}_{ij} :=(1-vt)^{w_1-w_i}\Phi_{ij}.
\]
By \eqref{eq:v-adic AT}, we have
$\widehat{\psi}_{ij}^{(-1)} =\widehat{\Phi}_{ij}\widehat{\psi}_{ij}$. Consider the block
system
\[
    \widehat{\psi}:=
    \begin{pmatrix}
        \widehat{\psi}_{11}\\
        \vdots\\
        \widehat{\psi}_{1m_1}\\
        \widehat{\psi}_{21}\\
        \vdots\\
        \widehat{\psi}_{sm_s}
    \end{pmatrix},
    \qquad \widehat{\Phi}:= \operatorname{diag} \left(\widehat{\Phi}_{11},\ldots,
    \widehat{\Phi}_{1m_1}, \widehat{\Phi}_{21},\ldots, \widehat{\Phi}_{sm_s}\right).
\]
Then $\widehat{\psi}^{(-1)} =\widehat{\Phi}\widehat{\psi}$. As in
\cref{cor:ABP-admissible}, the pair $(\widehat{\Phi},\widehat{\psi})$ satisfies the
hypotheses of \cref{thm:vadic-ABP} at $t=v^{-1}$.

For each $(i,j)$, put $r_{ij}:=\dep(\bs_{ij})$ and
\[
    b_{ij}:=
    \begin{cases}
        c_0,&(i,j)=(1,1),\\
        0,&\text{otherwise}.
    \end{cases}
\]
Define
\[
    \rho_{ij}:= \left(b_{ij},0,\ldots,0 \mid c_{ij}\boldsymbol{\eta}_v
    \right)
    \in\Mat_{1\times(1+r_{ij}d)}(\ov{K})
\]
and put
\[
    \rho:= \left(\rho_{11}\mid\cdots\mid\rho_{1m_1} \mid\rho_{21}\mid\cdots\mid\rho_{sm_s}\right).
\]
By \cref{lem:evaluation-row}, we have
\[
    \rho\widehat{\psi}(v^{-1}) = \widetilde{\pi}_v^{w_1}
    \left(c_0+\sum_{i=1}^s\sum_{j=1}^{m_i}c_{ij}Z_{ij}
    \right)
    =0.
\]
Hence, by \cref{thm:vadic-ABP}, there exists
\[
    \mathbf p(t)= \left(\mathbf p_{11}(t)\mid\cdots\mid\mathbf p_{1m_1}(t) \mid\mathbf
    p_{21}(t)\mid\cdots\mid\mathbf p_{sm_s}(t)\right),
\]
with entries in $\ov{K}[t]$, such that
\begin{equation}\label{eq:cross-lift}
    \mathbf p(v^{-1})=\rho,
    \qquad \mathbf p(t)\widehat{\psi}(t)=0,
\end{equation}
where $\mathbf p_{ij}(t)$ is the block corresponding to $\widehat{\psi}_{ij}$. Let
$Q_{ij}(t)\in\Mat_{1\times d}(\ov{K}[t])$ denote the last $d$ entries of
$\mathbf p_{ij}(t)$. By \eqref{eq:cross-lift}, we have
\[
    Q_{ij}(v^{-1})=c_{ij}\boldsymbol{\eta}_v.
\]
Choose $j_0$ with $c_{1j_0}\neq0$. Then $Q_{1j_0}(v^{-1})\neq0$, so $Q_{1j_0}(t)\neq0$.
By \cref{lem:orbit-nondegeneracy}, we have
\[
    Q_{1j_0}(v^{-q^N})\mathbf V_N\neq0
\]
for all sufficiently large $N$.

Fix such an $N$. By \eqref{eq:Omega-vanishing}, $\widehat{\psi}_{ij}(v^{-q^N})=0$
whenever $i>1$. For $i=1$, all entries of $\widehat{\psi}_{1j}(v^{-q^N})$ except the
last $d$ vanish. Thus, by \cref{prop:fiber-specialization}, specializing the second
identity in \eqref{eq:cross-lift} at $t=v^{-q^N}$ gives
\[
    \lambda_{v,N}^{w_1} \left(\widetilde{\pi}_v^{w_1}\right)^{q^N} \sum_{j=1}^{m_1}
    \left(Q_{1j}(v^{-q^N})\mathbf V_N\right) Z_{1j}^{q^N}=0.
\]
Since $\lambda_{v,N}\neq0$ and $\widetilde{\pi}_v\neq0$, we obtain
\[
    \sum_{j=1}^{m_1} \left(Q_{1j}(v^{-q^N})\mathbf V_N\right)^{q^{-N}} Z_{1j}=0.
\]
The coefficient of $Z_{1j_0}$ is nonzero, so this is a nontrivial $\ov{K}$-linear
relation among $Z_{11},\ldots,Z_{1m_1}$, contradicting \cref{lem:fixed-weight}.
\end{proof}

We are now ready to prove the first main theorem.

\begin{proof}[Proof of \cref{thm:main}]
For each $w\geq1$, choose a $K$-basis $\Bcal_w$ of $\Lcal_{v,w}$ consisting of $v$-adic
Carlitz multiple polylogarithms at algebraic points in their domains of convergence. Put
\[
    \Bcal:=\{1\}\cup\bigcup_{w\geq1}\Bcal_w.
\]
By \cref{lem:cross-weight}, the set $\Bcal$ is linearly independent over $\ov{K}$. Since
$\Bcal$ spans $\Lcal_v$ over $K$ and $\overline{\Lcal}_v$ over $\ov{K}$, it is a
$K$-basis of $\Lcal_v$ and an $\ov{K}$-basis of $\overline{\Lcal}_v$.

Consequently, we obtain
\[
    \overline{\Lcal}_v = \ov{K}\oplus \bigoplus_{w\geq1}\overline{\Lcal}_{v,w},
\]
and the canonical map
\[
    \ov{K}\otimes_K\Lcal_v \xrightarrow{\ \sim\ } \overline{\Lcal}_v
\]
is an isomorphism. Together with the harmonic product relations, this shows that
$\overline{\Lcal}_v$ is a graded $\ov{K}$-algebra.
\end{proof}

\subsection{\texorpdfstring{Proofs of Theorem~\ref{thm:mainMZV} and Corollary~\ref{cor:transcendence}}{Proofs of Theorem 1.3 and Corollary 1.4}}\label{sec:A2}

We now deduce the corresponding results for $v$-adic multiple zeta values. By
\cite[Theorem~5.2.5 and Definition~6.1.1]{CM21}, every $v$-adic multiple zeta value of
weight $w$ can be expressed as a $K$-linear combination of $v$-adic Carlitz multiple
star polylogarithmic values of the same weight at integral points. By
\cite[Remark~5.2.1]{CCM22}, these values belong to the $K$-span of convergent $v$-adic
Carlitz multiple star polylogarithmic values of weight $w$. Moreover, Chen
\cite[Proposition~3.3.5]{ChenABP} showed that the $K$-spans of convergent $v$-adic
Carlitz multiple polylogarithmic values and their star counterparts of weight $w$
coincide. Consequently, we have
\begin{equation}\label{eq:Z-in-L}
    \Zcal_{v,w}\subseteq\Lcal_{v,w},
    \qquad \overline{\Zcal}_{v,w} \subseteq\overline{\Lcal}_{v,w}.
\end{equation}

\begin{proof}[Proof of \cref{thm:mainMZV}]
For each $w\geq1$, choose a $K$-basis $\Bcal_w$ of $\Zcal_{v,w}$ consisting of $v$-adic
multiple zeta values of weight $w$. Put
\[
    \Bcal:=\{1\}\cup\bigcup_{w\geq1}\Bcal_w.
\]
By \eqref{eq:Z-in-L} and \cref{thm:main}, the set $\Bcal$ is linearly independent over
$\ov{K}$. Since $\Bcal$ spans $\Zcal_v$ over $K$ and $\overline{\Zcal}_v$ over $\ov{K}$,
it is a $K$-basis of $\Zcal_v$ and an $\ov{K}$-basis of $\overline{\Zcal}_v$.
Consequently, we obtain
\[
    \overline{\Zcal}_v = \ov{K}\oplus \bigoplus_{w\geq1}\overline{\Zcal}_{v,w},
\]
and the canonical map
\[
    \ov{K}\otimes_K\Zcal_v \xrightarrow{\ \sim\ } \overline{\Zcal}_v
\]
is an isomorphism. Together with the $q$-shuffle product relations, this shows that
$\overline{\Zcal}_v$ is a graded $\ov{K}$-algebra.
\end{proof}

We conclude with the proof of the transcendence statement.

\begin{proof}[Proof of \cref{cor:transcendence}]
Let $\bs\in\Ical$ be non-empty and let
$\bu\in(\ov{K}^{\times})^{\dep(\bs)} \cap\DD_{\bs,v}^{\Conv}$. Suppose that
$\Li_{\bs}(\bu)_v$ is nonzero and algebraic over $K$. Then by \cref{thm:main}, we get
\[
    \Li_{\bs}(\bu)_v \in \ov{K}\cap\overline{\Lcal}_{v,\wt(\bs)} = \{0\},
\]
which is a contradiction. Thus $\Li_{\bs}(\bu)_v$ is either zero or transcendental over
$K$. The same argument, using \cref{thm:mainMZV}, proves the desired result for
$\zeta_A(\bs)_v$.
\end{proof}

\sloppy

\end{document}